\documentclass[a4, 12pt]{amsart}
\usepackage{amssymb}
\usepackage{amstext}
\usepackage{amsmath}
\usepackage{amscd}
\usepackage{amsthm}
\usepackage{amsfonts}
\usepackage{enumerate}
\usepackage{graphicx}
\usepackage{latexsym}

\theoremstyle{plain}
\newtheorem{thm}{Theorem}[section]
\newtheorem*{thm*}{Theorem}
\newtheorem*{cor*}{Corollary}
\newtheorem*{prop*}{Proposition}

\newtheorem{prop}[thm]{Proposition}
\newtheorem{lem}[thm]{Lemma}

\newtheorem{cor}[thm]{Corollary}

\newtheorem*{claim*}{Claim}

\theoremstyle{definition}
\newtheorem{defn}[thm]{Definition}
\newtheorem{ex}[thm]{Example}
\newtheorem{rem}[thm]{Remark}

\newtheorem*{conj*}{Conjecture}

\newtheorem{note}[thm]{Notation}

\theoremstyle{remark}

\newtheorem*{tpf}{{\it Proof of Theorem \ref{main}}}

\numberwithin{equation}{thm}
\def\Hom{\operatorname{Hom}}
\def\lhom{\operatorname{\underline{Hom}}}

\def\Ext{\operatorname{Ext}}
\def\Tor{\operatorname{Tor}}

\def\mod{\operatorname{mod}}

\def\G{{\mathcal G}}

\def\Coker{\operatorname{Coker}}
\def\Ker{\operatorname{Ker}}

\def\tr{\operatorname{Tr}}

\def\rank{\operatorname{rank}}

\def\m{\mathfrak m}

\def\P{\mathfrak P}

\def\depth{\operatorname{depth}}

\def\pd{\operatorname{pd}}

\def\grade{\operatorname{grade}}

\def\A{{\mathcal A}}
\def\B{{\mathcal B}}
\def\C{{\mathcal C}}

\def\F{{\mathcal F}}
\def\X{{\mathcal X}}

\def\xx{\text{\boldmath $x$}}

\begin{document}

\setlength{\baselineskip}{15pt}

\title[Contravariantly finite resolving subcategories]{Contravariantly finite resolving subcategories over commutative rings}
\author{Ryo Takahashi}
\address{Department of Mathematical Sciences, Faculty of Science, Shinshu University, 3-1-1 Asahi, Matsumoto, Nagano 390-8621, Japan}
\email{takahasi@math.shinshu-u.ac.jp}
\begin{abstract}
Contravariantly finite resolving subcategories of the category of finitely generated modules have been playing an important role in the representation theory of algebras.
In this paper we study contravariantly finite resolving subcategories over commutative rings.
The main purpose of this paper is to classify contravariantly finite resolving subcategories over a henselian Gorenstein local ring; in fact there exist only three ones.
Our method to obtain this classification also recovers as a by-product the theorem of Christensen, Piepmeyer, Striuli and Takahashi concerning the relationship between the contravariant finiteness of the full subcategory of totally reflexive modules and the Gorenstein property of the base ring.
\end{abstract}
\maketitle
\section{Introduction}

Tilting theory has been a central topic in the representation theory of algebras.
This theory was originally introduced in the study of module categories over finite-dimensional algebras, and is now an important notion in a lot of areas of mathematics, including finite and algebraic group theory, commutative and noncommutative algebraic geometry, and algebraic topology; see \cite{hand} for the details.
Cohen-Macaulay approximation theory due to Auslander and Buchweitz \cite{ABu} is an aspect of tilting theory for commutative algebra.
Dualizing modules, which play a critical role in Cohen-Macaulay approximation theory, are special types of cotilting module.

The notion of a contravariantly finite subcategory (of the category of finitely generated modules), which is also called a (pre)covering class, was first introduced over artin algebras by Auslander and Smal{\o} \cite{AS} in connection with studying the problem of which subcategories admit almost split sequences.
The notion of a resolving subcategory was introduced by Auslander and Bridger \cite{ABr} in the study of totally reflexive modules, which are also called modules of Gorenstein dimension zero or finitely generated Gorenstein projective modules.
There is an application of contravariantly finite resolving subcategories to the study of the finitistic dimension conjecture \cite{AR}.
For other details of contravariantly finite subcategories and resolving subcategories, see \cite{aus0,AS2,AR1,AR2,ASo,Marz,KS,MR,Belig,Reiten} in addition to the articles introduced above.
On the other hand, there are a lot of articles which study contravariantly finite subcategories of the category of all modules; see \cite{Enochs,EJ,ATT,BEE} for example.

It was found out by Auslander and Reiten \cite{AR} that the notion of a contravariantly finite resolving subcategory is closely related to tilting theory; they classified contravariantly finite resolving subcategories over an artin algebra of finite global dimension in terms of cotilting modules.

\begin{thm}[Auslander-Reiten]\label{ausrei}
Let $R$ be an artin algebra of finite global dimension.
Then the assignment $T\mapsto{}^\perp T$ makes a bijection from the set of isomorphism classes of basic cotilting $R$-modules to the set of contravariantly finite resolving subcategories of $\mod R$.
\end{thm}

\noindent
Here, $\mod R$ denotes the category of finitely generated $R$-modules, and ${}^\perp T$ the full subcategory of $\mod R$ consisting of all finitely generated modules $M$ with $\Ext_R^i(M,T)=0$ for $i>0$.

This paper deals with contravariantly finite resolving subcategories over commutative rings.
The main result of this paper is the following classification theorem.

\begin{thm}\label{goren}
Let $R$ be a commutative henselian (e.g. complete) Gorenstein local ring.
Then all the contravariantly finite resolving subcategories of $\mod R$ are:
\begin{itemize}
\item
the full subcategory of free modules,
\item
the full subcategory of maximal Cohen-Macaulay modules, and
\item
$\mod R$ itself.
\end{itemize}
\end{thm}

The assumption of finiteness of global dimension in Theorem \ref{ausrei} is essential to obtain the bijection in the theorem.
Without this assumption, in general there exist so many contravariantly finite resolving subcategories, and it is extremely difficult to classify them; see \cite{MR}.
Theorem \ref{goren} should be remarkable since in it the ring $R$ is not assumed to have finite global dimension.

By the way, over an arbitrary commutative noetherian local ring, all the three subcategories in Theorem \ref{goren} are resolving, and the first and third subcategories are contravariantly finite.
The second subcategory is contravariantly finite if the base ring is a Cohen-Macaulay local ring admitting a dualizing module.
This fact is known as the Cohen-Macaulay approximation theorem due to Auslander and Buchweitz \cite{ABu}.

From now on, we explain our main results more minutely.
Throughout the rest of this section, let $R$ be a commutative noetherian henselian local ring.
We will give a sufficient condition for a given resolving subcategory to be a test for injective dimension:

\begin{thm}\label{mainlem}
Let $\X$ be a resolving subcategory of $\mod R$ such that the residue field of $R$ has a right $\X$-approximation.
Assume that there exists an $R$-module $G\in\X$ of infinite projective dimension with $\Ext_R^i(G,R)=0$ for $i\gg 0$.
Let $M$ be an $R$-module such that each $X\in\X$ satisfies $\Ext_R^i(X,M)=0$ for $i\gg 0$.
Then $M$ has finite injective dimension.
\end{thm}

Using this result, we will prove the following theorem.
Theorem \ref{goren} will be obtained from this.

\begin{thm}\label{cohen}
Let $\X\ne\mod R$ be a contravariantly finite resolving subcategory of $\mod R$.
Suppose that there is an $R$-module $G\in\X$ of infinite projective dimension such that $\Ext_R^i(G,R)=0$ for $i\gg 0$.
Then $R$ is Cohen-Macaulay and $\X$ consists of all maximal Cohen-Macaulay $R$-modules.
\end{thm}

On the other hand, Theorem \ref{mainlem} also yields the following corollary as an immediate consequence.
The assertion of this corollary is a main result of \cite{CPST}.
(Our method to get the corollary is quite different from the proof given in \cite{CPST}.)

\begin{cor}[Christensen-Piepmeyer-Striuli-Takahashi]\label{ccpst}
Suppose that there is a nonfree totally reflexive $R$-module.
If the full subcategory of $\mod R$ consisting of all totally reflexive $R$-modules is contravariantly finite, then $R$ is Gorenstein.
\end{cor}

A totally reflexive module was defined by Auslander \cite{aus1} as a common generalization of a free module and a maximal Cohen-Macaulay module over a Gorenstein local ring.
Auslander and Bridger \cite{ABr} proved that the full subcategory of totally reflexive modules over a left and right noetherian ring is resolving.
Many other properties of totally reflexive modules are stated in \cite{ABr} and \cite{chri}.

If $R$ is Gorenstein, then the totally reflexive $R$-modules are precisely the maximal Cohen-Macaulay $R$-modules, and so the full subcategory of totally reflexive $R$-modules is contravariantly finite by virtue of the Cohen-Macaulay approximation theorem.
Thus, Corollary \ref{ccpst} can be viewed as the converse of this fact.
Corollary \ref{ccpst} implies a geometric result: let $R$ be a homomorphic image of a regular local ring.
Suppose that there is a nonfree totally reflexive $R$-module and are only finitely many nonisomorphic indecomposable totally reflexive $R$-modules.
Then $R$ is an isolated simple hypersurface singularity.
For the details, see \cite{CPST}.

\section*{Conventions}

In the rest of this paper, we assume that all rings are commutative and noetherian, and that all modules are finitely generated.
Unless otherwise specified, let $R$ be a henselian local ring.
The unique maximal ideal of $R$ and the residue field of $R$ are denoted by $\m$ and $k$, respectively.
We denote by $\mod R$ the category of finitely generated $R$-modules.
By a {\em subcategory} of $\mod R$, we always mean a full subcategory of $\mod R$ which is closed under isomorphism.
Namely, in this paper, a subcategory $\X$ of $\mod R$ means a full subcategory such that every $R$-module which is isomorphic to some $R$-module in $\X$ is also in $\X$.

\section{Contravariant finiteness of totally reflexive modules}

In this section, we will state background materials which motivate the main results of this paper.
We start by recalling the definition of a totally reflexive module.

\begin{defn}
We denote by $(-)^\ast$ the $R$-dual functor $\Hom_R(-,R)$.
An $R$-module $M$ is called {\em totally reflexive} (or {\em of Gorenstein dimension zero}) if
\begin{enumerate}[\rm (1)]
\item
the natural homomorphism $M\to M^{\ast\ast}$ is an isomorphism, and
\item
$\Ext_R^i(M,R)=\Ext_R^i(M^\ast,R)=0$ for any $i>0$.
\end{enumerate}
\end{defn}

We introduce three subcategories of $\mod R$ which will often appear throughout this paper.

\begin{note}
We denote by $\F(R)$ the subcategory of $\mod R$ consisting of all free $R$-modules, by $\G(R)$ the subcategory of $\mod R$ consisting of all totally reflexive $R$-modules, and by $\C(R)$ the subcategory of $\mod R$ consisting of all maximal Cohen-Macaulay $R$-modules.
\end{note}

\begin{rem}\label{basic}
The following are basic properties of the subcategories $\F(R)$, $\G(R)$ and $\C(R)$.
\begin{enumerate}[\rm (1)]
\item
Every free $R$-module is totally reflexive, namely, $\F(R)$ is contained in $\G(R)$.
\item
If $R$ is Cohen-Macaulay, then all totally reflexive $R$-modules are maximal Cohen-Macaulay, namely, $\G(R)$ is contained in $\C(R)$.
\item
If $R$ is Gorenstein, then the totally reflexive $R$-modules are precisely the maximal Cohen-Macaulay $R$-modules, namely, $\G(R)$ coincides with $\C(R)$.
\end{enumerate}
The first and third assertions are immediate from the definition of a totally reflexive module.
The second assertion follows from \cite[Theorem (1.4.8)]{chri}.
\end{rem}

Next, we recall the notion of a right approximation over a subcategory of $\mod R$.

\begin{defn}
Let $\X$ be a subcategory of $\mod R$.
\begin{enumerate}[\rm (1)]
\item
Let $\phi:X\to M$ be a homomorphism of $R$-modules with $X\in\X$.
We say that $\phi$ is a {\em right $\X$-approximation} (of $M$) if the induced homomorphism $\Hom_R(X',\phi):\Hom_R(X',X)\to\Hom_R(X',M)$ is surjective for any $X'\in\X$.
\item
We say that $\X$ is {\em contravariantly finite} (in $\mod R$) if every $R$-module has a right $\X$-approximation.
\end{enumerate}
\end{defn}

By definition, contravariant finiteness can be regarded as a weaker version of the property that the inclusion functor has a right adjoint.

\begin{ex}\label{freecov}
Let $R$ be an arbitrary ring and $n$ a nonnegative integer.
The following subcategories of $\mod R$ are contravariantly finite, assuming in addition that $R$ is artinian in (7) and (8).
\begin{enumerate}[\rm (1)]
\item
$\mod R$.
\item
$\F(R)$.
\item
$\C(R)$, provided that $R$ is a Cohen-Macaulay local ring admitting a dualizing module.
\item
A subcategory of $\mod R$ closed under finite direct sums and direct summands in which there are only finitely many isomorphism classes of indecomposable $R$-modules.
\item
The subcategory of $\mod R$ consisting of all finitely generated $S$-modules, where $S$ is a fixed module-finite $R$-algebra.
\item
The subcategory of $\mod R$ consisting of all $M$ satisfying $\Ext_R^1(M,R)=0$.
\item
The subcategory of $\mod R$ consisting of all $M$ such that there exists an exact sequence $0\to M\to P^0\to\cdots\to P^{n-1}$ with each $P^i$ projective.
\item
The subcategory of $\mod R$ consisting of all $M$ such that there exists an exact sequence $I_{n-1}\to\cdots\to I_0\to M\to 0$ with each $I_i$ injective.
\end{enumerate}
The subcategory (1) is contravariantly finite since the identity map of every $R$-module is a right $\mod R$-approximation.
As to the fact that (2) is contravariantly finite, for a given $R$-module $M$, a surjective homomorphism from a module in $\F(R)$ to $M$ (i.e. a free cover of $M$) is exactly a right $\F(R)$-approximation of $M$.
As for the facts that the subcategories (3)-(8) are contravariantly finite, see \cite[Theorem 1.1]{ABu}, \cite[Proposition 4.2]{AS} (or \cite[Proof of Theorem 1.4]{dep0}), \cite[Theorem 1.3.1(1)]{Iyama}, \cite[Proposition (2.21)]{ABr} (or \cite[Theorem 7.1]{AS2} or \cite[Corollary 5.8]{tor}), \cite[Theorem 1.2]{AR3} and \cite[Theorem 1.8]{AR2} , respectively.
See also \cite[Theorems 4.3 and 4.4]{HI}.
\end{ex}

The following result is immediately obtained from Example \ref{freecov}(3) and Remark \ref{basic}(3).

\begin{thm}
If $R$ is Gorenstein, then $\G(R)$ is contravariantly finite.
\end{thm}

Recently, it has been proved that the (essential) converse of this theorem holds:

\begin{thm}\cite{CPST}\label{cpst}
Suppose that there is a nonfree totally reflexive $R$-module.
If $\G(R)$ is contravariantly finite, then $R$ is Gorenstein.
\end{thm}

Yoshino \cite{yos1,yos2} and Takahashi \cite{dep0,dep1,dep2,catgp} proved the statement of Theorem \ref{cpst} in several special cases.
Theorem \ref{cpst} is shown in \cite{CPST} by choosing a good maximal $R$-regular sequence $\xx=x_1,\dots,x_t$ and considering the smallest extension closed additive subcategory of $\mod(R/\xx R)$ containing all $R/\xx R$-modules $X/\xx X$ with $X\in\G(R)$.
Theorem \ref{cpst} will be recovered by Theorem \ref{main} stated in the next section, and the proof of Theorem \ref{main} is quite different from that of Theorem \ref{cpst} given in \cite{CPST}.

By the way, the assumption in Theorem \ref{cpst} of existence of a nonfree totally reflexive module is necessary.
Indeed, let $R$ be a non-Gorenstein Golod local ring (e.g. a non-Gorenstein Cohen-Macaulay local ring with minimal multiplicity).
Then $\G(R)$ coincides with $\F(R)$ by \cite[Examples 3.5(2)]{AM}.
In particular, $\G(R)$ is contravariantly finite by Example \ref{freecov}(2).

Theorem \ref{cpst} and Example \ref{freecov}(4) yield the following corollary, which improves \cite[Theorem 1.3]{igp}.

\begin{cor}
Let $R$ be a non-Gorenstein local ring.
If there is a nonfree totally reflexive $R$-module, then there are infinitely many nonisomorphic indecomposable totally reflexive $R$-modules.
\end{cor}

Combining this with \cite[Theorems (8.15) and (8.10)]{yos0} (cf. \cite[Satz 1.2]{herzog} and \cite[Theorem B]{BGS}), one obtains the following result.

\begin{cor}\label{cpst2}
Let $R$ be a homomorphic image of a regular local ring.
Suppose that there is a nonfree totally reflexive $R$-module but there are only finitely many nonisomorphic indecomposable totally reflexive $R$-modules.
Then $R$ is a simple hypersurface singularity.
\end{cor}

\section{Contravariantly finite resolving subcategories}

In this section, we will give the main theorems of this paper.
The most general result is Theorem \ref{main}, which implies all of the other results obtained in this paper.
First of all, we recall the definition of the syzygies of a given module.
Let $M$ be an $R$-module and $n$ a positive integer.
Let
$$
F_\bullet =(\cdots \overset{d_{n+1}}{\to} F_n \overset{d_n}{\to} F_{n-1} \overset{d_{n-1}}{\to} \cdots \overset{d_2}{\to} F_1 \overset{d_1}{\to} F_0 \to 0)
$$
be a minimal free resolution of $M$.
We define the {\em $n$th syzygy} $\Omega^nM$ of $M$ as the image of the homomorphism $d_n$.
We set $\Omega^0M=M$.
Note that the $n$th syzygy $\Omega^nM$ of $M$ is uniquely determined up to isomorphism, since so is a minimal free resolution of $M$.

Now we recall the definition of a resolving subcategory.

\begin{defn}\label{resolv}
A subcategory $\X$ of $\mod R$ is called {\em resolving} if it satisfies the following four conditions.
\begin{enumerate}[\rm (1)]
\item
$\X$ contains $R$.
\item
$\X$ is closed under direct summands: if $M$ is an $R$-module in $\X$ and $N$ is a direct summand of $M$, then $N$ is also in $\X$.
\item
$\X$ is closed under extensions: for an exact sequence $0 \to L \to M \to N \to 0$ of $R$-modules, if $L$ and $N$ are in $\X$, then $M$ is also in $\X$.
\item
$\X$ is closed under kernels of epimorphisms: for an exact sequence $0 \to L \to M \to N \to 0$ of $R$-modules, if $M$ and $N$ are in $\X$, then $L$ is also in $\X$.
\end{enumerate}
\end{defn}

A resolving subcategory is a subcategory such that any two ``minimal'' resolutions of a module by modules in it have the same length; see \cite[(3.12)]{ABr}.

\begin{rem}\label{resbas}
Let $\X$ be a resolving subcategory of $\mod R$.
Then
\begin{enumerate}[\rm (1)]
\item
$\X$ is closed under finite direct sums: if $M_1,\dots,M_n$ are a finite number of $R$-modules in $\X$, then so is the direct sum $M_1\oplus\cdots\oplus M_n$.
\item
$\X$ is closed under syzygies: if $M$ is an $R$-module in $\X$, then so is $\Omega^nM$ for every $n\ge 0$.
\end{enumerate}
Indeed, as to (1), it is enough to check that if $M$ and $N$ are $R$-modules in $\X$, then so is $M\oplus N$.
This follows by applying Definition \ref{resolv}(3) to the natural split exact sequence $0 \to M \to M\oplus N \to N \to 0$.
As to (2), it suffices to show that if $M$ is an $R$-module in $\X$, then so is the first syzygy $\Omega M$.
There is an exact sequence $0 \to \Omega M \to F \to M \to 0$ with $F$ free, and $F$ is in $\X$ by Definition \ref{resolv}(1) and the first assertion of this remark.
Hence $\Omega M$ is in $\X$ by Definition \ref{resolv}(4).
\end{rem}

\begin{ex}\label{rslvx}
Let $I$ be an ideal of $R$, $K$ an $R$-module, and $n$ a nonnegative integer.
The following subcategories of $\mod R$ are resolving.
\begin{enumerate}[\rm (1)]
\item
$\mod R$.
\item
$\F(R)$.
\item
$\G(R)$.
\item
$\{\,M\mid\grade(I,M)\ge\grade(I,R)\,\}$.
\item
$\C(R)$, provided that $R$ is Cohen-Macaulay.
\item
$\{\,M\mid M\text{ has bounded Betti numbers}\,\}$.
\item
$\{\,M\mid M\text{ has finite complexity}\,\}$.
\item
$\{\,M\mid M\text{ has lower complete intersection dimension zero}\,\}$.
\item
$\{\,M\mid\Tor_i^R(M,K)=0\text{ for }i>n\,\}$.
\item
$\{\,M\mid\Ext_R^i(M,K)=0\text{ for }i>n\,\}$.
\item
$\{\,M\mid\Tor_i^R(M,K)=0\text{ for }i\gg 0\,\}$.
\item
$\{\,M\mid\Ext_R^i(M,K)=0\text{ for }i\gg 0\,\}$.
\end{enumerate}
The fact that the subcategory (3) is resolving follows from \cite[(3.11)]{ABr} or \cite[Lemma 2.3]{AM}.
The resolving property of (5) follows from that of (4).
The fact that (7) and (8) are resolving follows from \cite[Proposition 4.2.4]{Avramov98} and \cite[Lemma 6.3.1]{Avramov02}, respectively.
The resolving properties of the other subcategories can be checked directly.
Note that the subcategories (11) and (12) are thick.
(A subcategory $\X$ of $\mod R$ is called {\em thick} provided that for any exact sequence $0 \to L \to M \to N \to 0$ in $\mod R$, if two of $L,M,N$ are in $\X$, then so is the third.)
\end{ex}

Now we state the most general result in this paper.

\begin{thm}\label{main}
Let $\X$ be a resolving subcategory of $\mod R$ such that the residue field $k$ has a right $\X$-approximation.
Assume that there exists an $R$-module $G\in\X$ of infinite projective dimension such that $\Ext_R^i(G,R)=0$ for $i\gg 0$.
Let $M$ be an $R$-module such that each $X\in\X$ satisfies $\Ext_R^i(X,M)=0$ for $i\gg 0$.
Then $M$ has finite injective dimension.
\end{thm}

Let $\X$ be as in Theorem \ref{main}.
Let $M$ be an $R$-module.
Theorem \ref{main} says that if $\Ext_R^{\gg 0}(X,M)=0$ for each $X\in\X$, then $\Ext_R^{\gg 0}(X,M)=0$ for each $X\in\mod R$.
In this sense, such a subcategory $\X$ of $\mod R$ is ``large''.

We shall prove Theorem \ref{main} in the next section.
In the rest of this section, we will state and prove several results by using Theorem \ref{main}.
We begin with two corollaries which are immediately obtained.

\begin{cor}\label{bestcor}
Let $\X$ be a resolving subcategory of $\mod R$ which is contained in the subcategory $\{\,M\mid\Ext_R^i(M,R)=0\text{ for }i\gg 0\,\}$ of $\mod R$.
Suppose that in $\X$ there is an $R$-module of infinite projective dimension.
If $k$ has a right $\X$-approximation, then $R$ is Gorenstein.
\end{cor}

\begin{proof}
Each module $X$ in $\X$ satisfies $\Ext_R^i(X,R)=0$ for $i\gg 0$.
Hence Theorem \ref{main} implies that $R$ has finite injective dimension as an $R$-module.
\end{proof}

\begin{cor}\label{maincor}
Let $\X$ be one of the following.
\begin{enumerate}[\rm (1)]
\item
$\G(R)$.
\item
The subcategory $\{\,M\mid\Ext_R^i(M,R)=0\text{ for }i>n\,\}$ of $\mod R$, where $n$ is a nonnegative integer.
\item
The subcategory $\{\,M\mid\Ext_R^i(M,R)=0\text{ for }i\gg 0\,\}$ of $\mod R$.
\end{enumerate}
Suppose that in $\X$ there is an $R$-module of infinite projective dimension.
If $k$ has a right $\X$-approximation, then $R$ is Gorenstein.
\end{cor}

\begin{proof}
As we observed in Example \ref{rslvx}, $\X$ is a resolving subcategory of $\mod R$.
Since $\X$ is contained in the subcategory $\{\,M\mid\Ext_R^i(M,R)=0\text{ for }i\gg 0\,\}$, the assertion follows from Corollary \ref{bestcor}.
\end{proof}

\begin{rem}
Corollary \ref{maincor} implies Theorem \ref{cpst}.
Indeed, any nonfree totally reflexive module has infinite projective dimension by \cite[(1.2.10)]{chri}.
\end{rem}

For a subcategory $\X$ of $\mod R$, let $\X^\perp$ (respectively, ${}^\perp\X$) denote the subcategory of $\mod R$ consisting of all $R$-modules $M$ such that $\Ext_R^i(X,M)=0$ (respectively, $\Ext_R^i(M,X)=0$) for all $X\in\X$ and $i>0$.
Applying Wakamatsu's lemma to a resolving subcategory, we obtain the following lemma.
This will often be used in the rest of this paper.

\begin{lem}\label{wak}
Let $\X$ be a resolving subcategory of $\mod R$.
If an $R$-module $M$ has a right $\X$-approximation, then there is an exact sequence $0 \to Y \to X \to M \to 0$ of $R$-modules with $X\in\X$ and $Y\in\X^\perp$.
\end{lem}

\begin{proof}
Since $R$ is henselian, $M$ has a {\em minimal} right $\X$-approximation: there exists a right $\X$-approximation $\phi:X\to M$ such that every endomorphism $f:X\to X$ with $\phi=\phi f$ is an automorphism; see \cite[Corollary 2.5]{dep0}.
Take an element $m\in M$, and define a homomorphism $\rho:R\to M$ by $\rho(1)=m$.
As $\X$ is a resolving subcategory, $R$ is in $\X$.
The homomorphism $\rho$ factors through $\phi$, and we see that $m=\phi(x)$ for some $x\in X$.
Thus $\phi$ is surjective.
Put $Y=\Ker\phi$.
Wakamatsu's lemma (see \cite{wk} or \cite[Lemma 2.1.1]{Xu}) shows that $\Ext_R^1(X',Y)=0$ for any $X'\in\X$.
We have $\Ext_R^i(X',Y)\cong\Ext_R^1(\Omega^{i-1}X',Y)=0$ for any $X'\in\X$ and $i>0$ by Remark \ref{resbas}(2).
Therefore $Y$ is in $\X^\perp$.
\end{proof}

By using Lemma \ref{wak} and the theorem which was formerly called ``Bass' conjecture'', we obtain another corollary of Theorem \ref{main}.

\begin{cor}\label{sbmncr2}
Let $\X$ be a resolving subcategory of $\mod R$ such that $k$ has a right $\X$-approximation and that $k$ is not in $\X$.
Assume that there is an $R$-module $G\in\X$ with $\pd_RG=\infty$ and $\Ext_R^i(G,R)=0$ for $i\gg 0$.
Then $R$ is Cohen-Macaulay and $\dim R>0$.
\end{cor}

\begin{proof}
According to Lemma \ref{wak}, there is an exact sequence $0 \to Y \to X \to k \to 0$ with $X\in\X$ and $Y\in\X^\perp$.
It is seen from Theorem \ref{main} that the $R$-module $Y$ has finite injective dimension.
Since $k$ is not in $\X$, we have $Y\ne 0$.
By virtue of \cite[Corollary 9.6.2 and Remark 9.6.4(a)]{BH}, the ring $R$ is Cohen-Macaulay.
Assume that $\dim R=0$.
Then $Y$ is injective (cf. \cite[Theorem 3.1.17]{BH}), and so the exact sequence splits.
Hence $k$ is isomorphic to a direct summand of $X$, and $k$ is in $\X$ as $\X$ is closed under direct summands.
This contradiction shows that $R$ has positive Krull dimension.
\end{proof}

To give the next corollary of Theorem \ref{main}, we need the following three lemmas.

\begin{lem}\label{kx}
Let $\X$ be a contravariantly finite resolving subcategory of $\mod R$.
Then, $k\in\X$ if and only if $\X=\mod R$.
\end{lem}

\begin{proof}
The ``if'' part is clear.
Suppose that $k$ is in $\X$ and let us show that $\X$ coincides with $\mod R$.
Fix an $R$-module $M$.
Then Lemma \ref{wak} yields an exact sequence $0 \to Y \to X \to M \to 0$ of $R$-modules with $X\in\X$ and $Y\in\X^\perp$.
We have $\Ext_R^i(k,Y)=0$ for all $i>0$, as $k\in\X$ and $Y\in\X^\perp$.
By \cite[Proposition 3.1.14]{BH}, $Y$ is an injective $R$-module, and so the exact sequence splits.
Hence $M$ is isomorphic to a direct summand of $X$.
Since $\X$ is closed under direct summands, $M$ is in $\X$.
This completes the proof of the lemma.
\end{proof}

\begin{lem}\label{prpprp}
Let $\X$ be a resolving subcategory of $\mod R$.
Suppose that every $R$-module in ${}^\perp(\X^\perp)$ admits a right $\X$-approximation.
Then $\X={}^\perp(\X^\perp)$.
\end{lem}

\begin{proof}
It is easy to see that $\X$ is contained in ${}^\perp(\X^\perp)$.
Let $M$ be an $R$-module in ${}^\perp(\X^\perp)$.
Then there is an exact sequence $0 \to Y \to X \to M \to 0$ of $R$-modules with $X\in\X$ and $Y\in\X^\perp$ by Lemma \ref{wak}.
This exact sequence can be regarded as an element of $\Ext_R^1(M,Y)$, and this Ext module vanishes since $M\in {}^\perp(\X^\perp)$ and $Y\in\X^\perp$.
Hence the exact sequence splits, and $M$ is isomorphic to a direct summand of $X$.
Since $\X$ is closed under direct summands, $M$ is in $\X$.
Thus $\X={}^\perp(\X^\perp)$.
\end{proof}

\begin{lem}\label{pdidsup}
Let $M$ and $N$ be nonzero $R$-modules.
Assume either that $M$ has finite projective dimension or that $N$ has finite injective dimension.
Then one has an equality
$$
\sup\{\,i\mid\Ext_R^i(M,N)\ne 0\,\}=\depth R-\depth_RM.
$$
\end{lem}

\begin{proof}
Assume that $\pd_RM=s<\infty$.
Then it is easy to see that $\Ext_R^i(M,N)=0$ for $i>s$, so that $\sup\{\,i\mid\Ext_R^i(M,N)\ne 0\,\}\le s$.
There is an exact sequence
$$
0 \to R^{\oplus n_s} \overset{A_s}{\to} R^{\oplus n_{s-1}} \overset{A_{s-1}}{\to} \cdots \overset{A_1}{\to} R^{\oplus n_0} \to M \to 0
$$
such that each $A_i$ is a matrix over $R$ all of whose entries are elements in the maximal ideal $\m$.
Dualizing this by $N$, we obtain an exact sequence $N^{\oplus n_{s-1}} \overset{{}^t(A_s)}{\to} N^{\oplus n_s} \to \Ext_R^s(M,N) \to 0$.
It is seen from Nakayama's lemma that $\Ext_R^s(M,N)\ne 0$.
Thus $\sup\{\,i\mid\Ext_R^i(M,N)\ne 0\,\}=s=\depth R-\depth_RM$, where the last equality follows from the Auslander-Buchsbaum formula.

As to the case where $N$ has finite injective dimension, see \cite[Satz 2.6]{ische} or \cite[Exercise 3.1.24]{BH}.
\end{proof}

The above lemma can also be shown by using the formulas \cite[Theorem 4.1]{Foxby} and \cite[Corollary (2.14)]{Christensen2}.

Now we can show the following corollary.
There are only two contravariantly finite resolving subcategories possessing such $G$ as in the corollary.

\begin{cor}\label{maincr2}
Let $\X$ be a contravariantly finite resolving subcategory of $\mod R$.
Assume that there is an $R$-module $G\in\X$ with $\pd_RG=\infty$ and $\Ext_R^i(G,R)=0$ for $i\gg 0$.
Then either of the following holds.
\begin{enumerate}[\rm (1)]
\item
$\X=\mod R$,
\item
$R$ is Cohen-Macaulay and $\X=\C(R)$.
\end{enumerate}
\end{cor}

\begin{proof}
Suppose that $\X\ne\mod R$.
Then Lemma \ref{kx} says that $k$ is not in $\X$.
By Corollary \ref{sbmncr2}, $R$ is Cohen-Macaulay.

First, we show that $\C(R)$ is contained in $\X$.
For this, let $M$ be a maximal Cohen-Macaulay $R$-module.
We have only to prove that $M$ is in ${}^\perp(\X^\perp)$ by Lemma \ref{prpprp}.
Let $N$ be a nonzero $R$-module in $\X^\perp$.
Theorem \ref{main} implies that $N$ is of finite injective dimension.
Since $M$ is maximal Cohen-Macaulay, we have $\sup\{\,i\mid\Ext_R^i(M,N)\ne 0\,\}=0$ by Lemma \ref{pdidsup}.
Therefore $\Ext_R^i(M,N)=0$ for all $N\in\X^\perp$ and $i>0$.
It follows that $M$ is in ${}^\perp(\X^\perp)$, as desired.

Next, we show that $\X$ is contained in $\C(R)$.
We have an exact sequence $0 \to Y \to X \to k \to 0$ with $X\in\X$ and $Y\in\X^\perp$ by Lemma \ref{wak}.
Since $k$ is not in $\X$, the module $Y$ is nonzero.
By Theorem \ref{main}, $Y$ has finite injective dimension.
According to Lemma \ref{pdidsup}, for a nonzero $R$-module $X'$ in $\X$, we have equalities $0\ge\sup\{\,i\mid\Ext_R^i(X',Y)\ne 0\,\}=\depth R-\depth_RX'=\dim R-\depth_RX'$.
Therefore $X'$ is a maximal Cohen-Macaulay $R$-module, as desired.
\end{proof}

Next, we study contravariantly finite resolving subcategories all of whose objects $X$ satisfy $\Ext_R^{\gg 0}(X,R)=0$.
We start by considering special ones among such subcategories.

\begin{prop}\label{pd}
Let $\X$ be a contravariantly finite resolving subcategory of $\mod R$.
Suppose that every $R$-module in $\X$ has finite projective dimension.
Then either of the following holds.
\begin{enumerate}[\rm (1)]
\item
$\X=\F(R)$,
\item
$R$ is regular and $\X=\mod R$.
\end{enumerate}
\end{prop}

\begin{proof}
If $\X=\mod R$, then our assumption says that all $R$-modules have finite projective dimension.
Hence $R$ is regular.
Assume that $\X\ne\mod R$.
Then there is an $R$-module $M$ which is not in $\X$.
There is an exact sequence $0 \to Y \to X \to M \to 0$ with $X\in\X$ and $Y\in\X^\perp$ by Lemma \ref{wak}.
Note that $Y\ne 0$ as $M\notin\X$.
Fix a nonzero $R$-module $X'\in\X$.
We have $\Ext_R^i(X',Y)=0$ for all $i>0$, and hence $\pd_RX'=\sup\{\,i\mid\Ext_R^i(X',Y)\ne 0\,\}=0$ by Lemma \ref{pdidsup} and the Auslander-Buchsbaum formula.
Hence $X'$ is free.
This means that $\X$ is contained in $\F(R)$.
On the other hand, $\X$ contains $\F(R)$ since $\X$ is resolving.
Therefore $\X=\F(R)$.
\end{proof}

Combining Proposition \ref{pd} with Corollary \ref{maincr2}, we can get the following.

\begin{cor}\label{216}
Let $\X$ be a contravariantly finite resolving subcategory of $\mod R$.
Suppose that every module $X\in\X$ is such that $\Ext_R^i(X,R)=0$ for $i\gg 0$.
Then one of the following holds.
\begin{enumerate}[\rm (1)]
\item
$\X=\F(R)$,
\item
$R$ is Gorenstein and $\X=\C(R)$,
\item
$R$ is Gorenstein and $\X=\mod R$.
\end{enumerate}
\end{cor}

\begin{proof}
The corollary follows from Proposition \ref{pd} in the case where all $R$-modules in $\X$ are of finite projective dimension.
So suppose that in $\X$ there exists an $R$-module of infinite projective dimension.
Then Corollary \ref{maincr2} shows that either of the following holds.
\begin{enumerate}[\rm (i)]
\item
$\X=\mod R$,
\item
$R$ is Cohen-Macaulay and $\X=\C(R)$.
\end{enumerate}
By the assumption that every $X\in\X$ satisfies $\Ext_R^i(X,R)=0$ for $i\gg 0$, we have $\Ext_R^i(k,R)=0$ for $i\gg 0$ in the case (i).
In the case (ii), since $\Omega ^dk$ is in $\X$ where $d=\dim R$, we have $\Ext_R^{i+d}(k,R)\cong\Ext_R^i(\Omega^dk,R)=0$ for $i\gg 0$.
Thus, in both cases, the ring $R$ is Gorenstein.
\end{proof}

Finally, we obtain the following result from Corollary \ref{216}, Example \ref{freecov}(2)(3) and Example \ref{rslvx}(2)(5).
It says that the category of finitely generated modules over a Gorenstein local ring possesses only three contravariantly finite resolving subcategories.

\begin{cor}
Let $R$ be a Gorenstein local ring.
Then all the contravariantly finite resolving subcategories of $\mod R$ are $\F(R)$, $\C(R)$ and $\mod R$.
\end{cor}

\section{Proof of Theorem \ref{main}}

This section is devoted to proving Theorem \ref{main}.
First of all, let us recall the definition and some fundamental properties of the transpose of a given module.
Let $M$ be an $R$-module.
Take a minimal free resolution $F_\bullet =(\cdots \overset{d_2}{\to} F_1 \overset{d_1}{\to} F_0 \to 0)$ of $M$.
We define the {\em transpose} $\tr M$ of $M$ as the cokernel of the $R$-dual homomorphism $d_1^\ast:F_0^\ast \to F_1^\ast$ of $d_1$.

\begin{rem}\label{trem}
Let $M$ be an $R$-module.
\begin{enumerate}[(1)]
\item
The transpose $\tr M$ of $M$ are uniquely determined up to isomorphism, since so is a minimal free resolution of $M$.
\item
One has $\tr(\tr M)\cong M$ and $M^\ast\cong\Omega^2\tr M$ up to free summand.
\end{enumerate}
\end{rem}

For an $R$-module $M$, we denote by $\nu_R(M)$ the minimal number of generators of $M$, i.e. $\nu_R(M)=\dim_k(M\otimes_Rk)$.

\begin{lem}\label{tr}
Let $M$ be an $R$-module.
The transpose $\tr M$ has no nonzero free summand.
\end{lem}

\begin{proof}
Let $F_\bullet = (\cdots \overset{d_2}{\to} F_1 \overset{d_1}{\to} F_0 \to 0)$ be a minimal free resolution of $M$.
Assume that $\tr M$ has a nonzero free summand.
Then there exists a surjective homomorphism $\pi:\tr M\to R$.
We obtain a commutative diagram
$$
\begin{CD}
@. 0 @. 0 \\
@. @VVV @VVV \\
F_0^\ast @>{\delta}>> F @>>> N @>>> 0 \\
@| @V{\theta}VV @VVV \\
F_0^\ast @>{d_1^\ast}>> F_1^\ast @>{\varepsilon}>> \tr M @>>> 0 \\
@. @V{\pi\varepsilon}VV @V{\pi}VV \\
@. R @= R \\
@. @VVV @VVV \\
@. 0 @. 0
\end{CD}
$$
with exact rows and columns, where $F$ is a free $R$-module.
This diagram yields isomorphisms $M=\Coker d_1\cong\Coker(\delta^\ast\theta^\ast)\cong\Coker(\delta^\ast)$, as the homomorphism $\theta^\ast$ is surjective.
Hence there exists a surjective homomorphism $F^\ast\to\Omega M$, and we have $\rank F=\rank F^\ast\ge\nu_R(\Omega M)=\rank F_1=\rank F_1^\ast=\rank F+1$, which is a contradiction.
\end{proof}

For an $R$-module $M$, let $M^\ast M$ be the ideal of $R$ generated by the subset
$$
\{\,f(x)\mid f\in M^\ast,x\in M\,\}
$$
of $R$.
We verify that this ideal has the following property.

\begin{lem}\label{indec}
An $R$-module $M$ has a nonzero free summand if and only if $M^\ast M=R$.
\end{lem}

\begin{proof}
If $M$ has a nonzero free summand, then there exists a surjective homomorphism $\pi:M\to R$.
Hence there is an element $x\in M$ such that $\pi(x)=1$.
We have $1=\pi(x)\in M^\ast M$, so $M^\ast M=R$.

Conversely, suppose that the equality $M^\ast M=R$ holds.
Then there are elements $f_1,\dots,f_n\in M^\ast$ and $x_1,\dots,x_n\in M$ such that $1=f_1(x_1)+\cdots+f_n(x_n)$.
Define homomorphisms $f:M^{\oplus n}\to R$ and $g:R\to M^{\oplus n}$ by $f(y_1,\dots,y_n)=f_1(y_1)+\cdots+f_n(y_n)$ for $y_1,\dots,y_n\in M$ and $g(a)=(ax_1,\dots,ax_n)$ for $a\in R$.
Then the composite map $fg$ is the identity map of $R$, so $g$ is a split monomorphism.
Hence $R$ is isomorphic to a direct summand of $M^{\oplus n}$.
Since $R$ is henselian, $R$ is isomorphic to a direct summand of $M$ by virtue of the Krull-Schmidt theorem, that is, $M$ has a nonzero free summand.
\end{proof}

Let $M$ and $N$ be $R$-modules.
We denote by $\P_R(M,N)$ the $R$-submodule of $\Hom_R(M,N)$ consisting of all homomorphisms from $M$ to $N$ which factors through some free $R$-module.
We denote by $\lhom_R(M,N)$ the residue $R$-module $\Hom_R(M,N)/\P_R(M,N)$.

For $R$-modules $M$ and $N$, there is a natural homomorphism
$$
\lambda_{MN}:M\otimes_RN\to\Hom_R(M^\ast,N)
$$
which is given by $\lambda_{MN}(x\otimes y)(f)=f(x)y$ for $x\in M$, $y\in N$ and $f\in M^\ast$.
The following proposition will be used to prove the essential part of the proof of Theorem \ref{main}.

\begin{prop}\label{key}
Let $\X$ be a subcategory of $\mod R$ and $0 \to Y \overset{f}{\to} X \to M \to 0$ an exact sequence of $R$-modules with $X\in\X$ and $Y\in\X^\perp$.
Let $G\in\X$, set $H=\tr\Omega G$, and suppose that $(H^\ast H)M=0$.
Let $0 \to K \overset{g}{\to} F \overset{h}{\to} H \to 0$ be an exact sequence of $R$-modules with $F$ free.
Then the induced sequence
$$
\begin{CD}
0 @>>> K\otimes_RY @>{g\otimes_RY}>> F\otimes_RY @>{h\otimes_RY}>> H\otimes_RY @>>> 0
\end{CD}
$$
is exact, and the map $h\otimes_RY$ factors through the map $F\otimes_Rf:F\otimes_RY\to F\otimes_RX$.
\end{prop}

\begin{proof}
According to \cite[Proposition (2.6)]{ABr}, there is a commutative diagram
$$
\begin{CD}
@. 0 \\
@. @VVV \\
@. \Ext_R^1(\tr H,Y) \\
@. @VVV \\
@. H\otimes_RY @>{\delta}>> H\otimes_RX @>{\varepsilon}>> H\otimes_RM @>>> 0 \\
@. @V{\alpha}VV @V{\beta}VV @V{\gamma}VV \\
0 @>>> \Hom_R(H^\ast,Y) @>{\zeta}>> \Hom_R(H^\ast,X) @>{\eta}>> \Hom_R(H^\ast,M) @>>> \Ext_R^1(H^\ast,Y) \\
@. @VVV \\
@. \Ext_R^2(\tr H,Y) \\
@. @VVV \\
@. 0
\end{CD}
$$
with exact rows and columns, where $\alpha=\lambda_{HY}$, $\beta=\lambda_{HX}$ and $\gamma=\lambda_{HM}$.
Note from Remark \ref{trem}(2) that $\tr H\cong \Omega G$ and $H^\ast\cong\Omega ^2\tr H\cong\Omega ^3G$ up to free summand.
Since $G\in\X$ and $Y\in\X^\perp$, we have $\Ext_R^1(\tr H,Y)\cong\Ext_R^2(G,Y)=0$, $\Ext_R^2(\tr H,Y)\cong\Ext_R^3(G,Y)=0$ and $\Ext_R^1(H^\ast,Y)\cong\Ext_R^4(G,Y)=0$.
Hence $\alpha$ is an isomorphism and $\eta$ is a surjective homomorphism.
Diagram chasing shows that $\delta$ is an injective homomorphism.
From the assumption that $(H^\ast H)M=0$, we see that $\gamma=\lambda_{HM}$ is the zero map.
We thus obtain the following commutative diagram with exact rows.
$$
\begin{CD}
0 @>>> H\otimes_RY @>{\delta}>> H\otimes_RX @>{\varepsilon}>> H\otimes_RM @>>> 0 \\
@. @V{\alpha}V{\cong}V @V{\beta}VV @V{\gamma}V{0}V \\
0 @>>> \Hom_R(H^\ast,Y) @>{\zeta}>> \Hom_R(H^\ast,X) @>{\eta}>> \Hom_R(H^\ast,M) @>>> 0
\end{CD}
$$
As $\eta\beta=\gamma\varepsilon=0$, there exists a homomorphism
$$
\sigma: H\otimes_RX \to \Hom_R(H^\ast,Y)
$$
with $\zeta\sigma=\beta$.
The commutativity $\zeta\alpha=\beta\delta$ and the injectivity of $\zeta$ imply that $\sigma\delta=\alpha$.
Put
$$
\tau=\alpha^{-1}\sigma:H\otimes_RX \to H\otimes_RY.
$$
The map $\delta$ is a split monomorphism with a splitting map $\tau$.
There is a commutative diagram
$$
\begin{CD}
@. 0 \\
@. @VVV \\
@. \Tor_1^R(H,Y) \\
@. @VVV \\
@. K\otimes_RY @>>> K\otimes_RX @>>> K\otimes_RM @>>> 0 \\
@. @V{g\otimes_RY}VV @VVV @VVV \\
0 @>>> F\otimes_RY @>{F\otimes_Rf}>> F\otimes_RX @>>> F\otimes_RM @>>> 0 \\
@. @V{h\otimes_RY}VV @V{\xi}VV @VVV \\
0 @>>> H\otimes_RY @>{\delta}>> H\otimes_RX @>{\varepsilon}>> H\otimes_RM @>>> 0 \\
@. @VVV @VVV @VVV \\
@. 0 @. 0 @. 0
\end{CD}
$$
with exact rows and columns, and we have $h\otimes_RY=(\tau\delta)(h\otimes_RY)=\tau\xi(F\otimes_Rf)$.
Thus, the homomorphism $h\otimes_RY$ factors through the homomorphism $F\otimes_Rf$.

We have an isomorphism $\Tor_1^R(H,Y)=\Tor_1^R(\tr\Omega G,Y)\cong\lhom_R(\Omega G,Y)$ by \cite[Lemma (3.9)]{yos0}.
It remains to prove that $\lhom_R(\Omega G,Y)=0$.
There is an exact sequence $0 \to \Omega G \overset{\theta}{\to} P \to G \to 0$ of $R$-modules such that $P$ is a free $R$-module.
Let $\rho\in\Hom_R(\Omega G,Y)$.
We want to show that this map factors through some free $R$-module.
There is a pushout diagram:
$$
\begin{CD}
0 @>>> \Omega G @>{\theta}>> P @>>> G @>>> 0 \\
@. @V{\rho}VV @V{\phi}VV @| \\
0 @>>> Y @>{\psi}>> Q @>>> G @>>> 0
\end{CD}
$$
The second row is a split exact sequence since it can be regarded as an element of $\Ext_R^1(G,Y)$, which vanishes as $G\in\X$ and $Y\in\X^\perp$.
Hence there is a homomorphism $\pi:Q\to Y$ such that $\pi\psi=1$, and we have $\rho=(\pi\psi)\rho=\pi\phi\theta$.
Therefore $\rho$ factors through the free $R$-module $P$, as desired.
\end{proof}

The proof of our next result will require the following elementary lemma.

\begin{lem}\label{elmntry}
Let $M$ be an $R$-module, and let
$$
M=M_0 \overset{f_0}{\to} M_1 \overset{f_1}{\to} M_2 \overset{f_2}{\to} \cdots
$$
be a sequence of surjective homomorphisms.
Then $f_i$ is an isomorphism for $i\gg 0$.
\end{lem}

\begin{proof}
For a positive integer $i$, let $K_i$ be the kernel of the composite map $g_i:=f_{i-1}f_{i-2}\cdots f_1f_0$.
Then there is an ascending chain $K_1\subseteq K_2\subseteq K_3\subseteq\cdots$ of submodules of $M$.
Since $M$ is a noetherian $R$-module, there exists an integer $n$ such that $K_i=K_{i+1}$ for any $i\ge n$.
For each positive integer $i$, we have a commutative diagram
$$
\begin{CD}
M/K_i @>{\cong}>> M_i \\
@V{\pi_i}VV @V{f_i}VV \\
M/K_{i+1} @>{\cong}>> M_{i+1}
\end{CD}
$$
where the horizontal isomorphisms are induced by $g_i$ and $g_{i+1}$, and $\pi$ is a natural surjective homomorphism.
Since $\pi_i$ is an isomorphism for $i\ge n$, the map $f_i$ is an isomorphism for $i\ge n$.
\end{proof}

Let $\A$ and $\B$ be abelian categories.
A {\em contravariant cohomological $\delta$-functor} $D$ from $\A$ to $\B$ is defined as a collection of additive contravariant functors $D^n:\A\to\B$ for each integer $n\ge 0$, together with morphisms $\delta^n(s):D^n(X)\to D^{n+1}(Z)$ for each short exact sequence $s:0 \to X \overset{f}{\to} Y \overset{g}{\to} Z \to 0$ in $\A$.
We make the convention that $D^n=0$ for $n<0$.
The following two conditions are imposed:
\begin{enumerate}[\rm (1)]
\item
For each short exact sequence as above, there is a long exact sequence
$$
\cdots \to D^{n-1}(X) \overset{\delta^{n-1}(s)}{\to} D^n(Z) \overset{D^n(g)}{\to} D^n(Y) \overset{D^n(f)}{\to} D^n(X) \overset{\delta^n(s)}{\to} D^{n+1}(Z) \to \cdots.
$$
\item
For each commutative diagram
$$
\begin{CD}
s\phantom{'}:\quad 0 @>>> X @>{f}>> Y @>{g}>> Z @>>> 0 \\
@. @V{\alpha}VV @V{\beta}VV @V{\gamma}VV \\
s':\quad 0 @>>> X' @>{f'}>> Y' @>{g'}>> Z' @>>> 0
\end{CD}
$$
with exact rows, the diagram
$$
\begin{CD}
D^n(X') @>{\delta^n(s')}>> D^{n+1}(Z') \\
@V{D^n(\alpha)}VV @V{D^{n+1}(\gamma)}VV \\
D^n(X) @>{\delta^n(s)}>> D^{n+1}(Z)
\end{CD}
$$
is commutative.
\end{enumerate}

Now we can prove the following, which will play a key role in the proof of Theorem \ref{main}.

\begin{prop}\label{keyprop}
Let $\X$ be a subcategory of $\mod R$ which is closed under syzygies.
Let $0 \to Y \to X \to M \to 0$ be an exact sequence of $R$-modules with $X\in\X$ and $Y\in\X^\perp$.
Suppose that there is an $R$-module $G\in\X$ with $\pd_RG=\infty$ and $\Ext_R^i(G,R)=0$ for $i\gg 0$.
Put $H_i=\tr\Omega(\Omega^iG)$ and assume that $((H_i)^\ast H_i)M=0$ for $i\gg 0$.
Let $D=(D^j)_{j\ge 0}:\mod R\to\mod R$ be a contravariant cohomological $\delta$-functor.
If $D^j(X)=0$ for $j\gg 0$, then $D^j(Y)=D^j(M)=0$ for $j\gg 0$.
\end{prop}

\begin{proof}
Since $\X$ is closed under syzygies, the module $\Omega^iG$ is in $\X$ for $i\ge 0$.
Replacing $G$ with $\Omega^iG$ for $i\gg 0$, we may assume that $\Ext_R^i(G,R)=0$ for all $i>0$ and that $((H_i)^\ast H_i)M=0$ for all $i\ge 0$.
Let 
$$
F_\bullet=(\cdots \overset{d_{i+1}}{\to} F_i \overset{d_i}{\to} F_{i-1} \overset{d_{i-1}}{\to} \cdots \overset{d_2}{\to} F_1 \overset{d_1}{\to} F_0 \to 0)
$$
be a minimal free resolution of $G$.
Dualizing this by $R$ yields an exact sequence
$$
0 \to G^\ast \to (F_0)^\ast \overset{(d_1)^\ast}{\to} (F_1)^\ast \overset{(d_2)^\ast}{\to} \cdots \overset{(d_i)^\ast}{\to} (F_i)^\ast \overset{(d_{i+1})^\ast}{\to} (F_{i+1})^\ast \overset{(d_{i+2})^\ast}{\to} (F_{i+2})^\ast \overset{(d_{i+3})^\ast}{\to} \cdots,
$$
and it is easily seen that $H_i\cong(\Omega^{i+3}G)^\ast$ and $\Omega H_i\cong(\Omega^{i+2}G)^\ast$ for $i\ge 0$.
By Proposition \ref{key}, for each integer $i\ge 0$ we have an exact sequence
$$
0 \to (\Omega^{i+2}G)^\ast\otimes_RY \to (F_{i+2})^\ast\otimes_RY \overset{f_i}{\to} (\Omega^{i+3}G)^\ast\otimes_RY \to 0
$$
such that $f_i$ factors through $(F_{i+2})^\ast\otimes_RX$.

There is an integer $a\ge 0$ such that $D^j(X)=0$ for all $j\ge a$.
From the above short exact sequence, we get a long exact sequence
$$
\cdots \overset{D^j(f_i)}{\to} D^j((F_{i+2})^\ast\otimes_RY) \to D^j((\Omega^{i+2}G)^\ast\otimes_RY) \to D^{j+1}((\Omega^{i+3}G)^\ast\otimes_RY) \overset{D^{j+1}(f_i)}{\to} \cdots.
$$
The homomorphism $D^j(f_i)$ factors through $D^j((F_{i+2})^\ast\otimes_RX)$, which vanishes for $j\ge a$ as $(F_{i+2})^\ast\otimes_RX$ is isomorphic to a direct sum of copies of $X$.
Hence $D^j(f_i)=0$ for $j\ge a$, and we obtain an exact sequence
$$
0 \to D^j((F_{i+2})^\ast\otimes_RY) \to D^j((\Omega^{i+2}G)^\ast\otimes_RY) \overset{\varepsilon_{i,j}}{\to} D^{j+1}((\Omega^{i+3}G)^\ast\otimes_RY) \to 0
$$
for $i\ge 0$ and $j\ge a$.
Thus, there is a sequence
$$
D^a((\Omega^2G)^\ast\otimes_RY) \overset{\varepsilon_{0,a}}{\to} D^{a+1}((\Omega^3G)^\ast\otimes_RY) \overset{\varepsilon_{1,a+1}}{\to} D^{a+2}((\Omega^4G)^\ast\otimes_RY) \overset{\varepsilon_{2,a+2}}{\to} \cdots
$$
of surjective homomorphisms of $R$-modules.
Lemma \ref{elmntry} says that there exists an integer $b\ge 0$ such that
$$
\varepsilon_{l,a+l}:D^{a+l}((\Omega^{l+2}G)^\ast\otimes_RY) \to D^{a+l+1}((\Omega^{l+3}G)^\ast\otimes_RY)
$$
is an isomorphism for every $l\ge b$.
It follows from the above short exact sequence that $D^{a+l}((F_{l+2})^\ast\otimes_RY)=0$ for every $l\ge b$.
Since $G$ has infinite projective dimension, each $F_i$ is a nonzero free $R$-module, and the module $(F_{l+2})^\ast\otimes_RY$ is isomorphic to a nonzero direct sum of copies of $Y$.
Thus we have $D^j(Y)=0$ for $j\ge a+b$.
From the exact sequence $0 \to Y \to X \to M \to 0$ we see that $D^j(M)=0$ for $j\ge a+b+1$.
\end{proof}

Now we are in the position to achieve the purpose of this section.

\begin{tpf}
Since $k$ admits a right $\X$-approximation, there exists an exact sequence $0 \to Y \to X \to k \to 0$ of $R$-modules with $X\in\X$ and $Y\in\X^\perp$ by Lemma \ref{wak}.
For an integer $i\ge 0$, put $H_i=\tr\Omega(\Omega^iG)$.
Lemma \ref{tr} says that $H_i$ has no nonzero free summand.
We have $(H_i)^\ast H_i\ne R$ by Lemma \ref{indec}.
Hence $((H_i)^\ast H_i)k=0$ for $i\ge 0$.
Applying Proposition \ref{keyprop} to the contravariant cohomological $\delta$-functor $D=(\Ext_R^j(\quad,M))_{j\ge 0}$, we obtain $D^j(k)=0$ for $j\gg 0$.
Namely, we have $\Ext_R^j(k,M)=0$ for $j\gg 0$, which implies that $M$ has finite injective dimension.
\qed
\end{tpf}

\section*{Acknowledgments}

The author would like to give his deep gratitude to Shiro Goto and Osamu Iyama for a lot of valuable discussions and helpful suggestions.
The author also thanks Mitsuyasu Hashimoto, Yuji Yoshino and an anonymous referee for important and useful comments.


\end{document}